\documentclass[10pt]{amsart}
\usepackage{amssymb,mathrsfs,graphicx,extpfeil}
\usepackage{epsfig}
\usepackage{indentfirst, latexsym, amssymb, enumerate,amsmath,graphicx,amsthm}
\usepackage{float}
\usepackage{colortbl}
\usepackage{epsfig,subfigure}

\title[Cucker-Smale model with randomly switching network]{On the stochastic flocking of the Cucker-Smale flock with randomly switching topologies}

\author[Dong]{Jiu-Gang Dong}
\address[Jiu-Gang Dong]{\newline Department of Mathematics, Harbin Institute of Technology,\newline Harbin 150001, P. R. China}
\email{jgdong@hit.edu.cn}

\author[Ha]{Seung-Yeal Ha}
\address[Seung-Yeal Ha]{\newline Department of Mathematical Sciences and Research Institute of Mathematics, \newline Seoul National University, Seoul 08826 \newline
and School of Mathematics, Korea Institute for Advanced Study,\newline
 Hoegiro 85, Seoul 02455, Korea (Republic of)}
\email{syha@snu.ac.kr}

\author[Jung]{Jinwook Jung}
\address[Jinwook Jung]{\newline Department of Mathematical Sciences, Seoul National University,\newline Seoul 08826, Korea (Republic of)}
\email{warp100@snu.ac.kr}

\author[Kim]{Doheon Kim}
\address[Doheon Kim]{\newline School of Mathematics, Korea Institute for Advanced Study,\newline Hoegiro 85, Seoul 02455, Korea (Republic of)}
\email{doheonkim@kias.re.kr}

\newtheorem{theorem}{Theorem}[section]
\newtheorem{lemma}{Lemma}[section]
\newtheorem{corollary}{Corollary}[section]
\newtheorem{proposition}{Proposition}[section]

\newtheorem{remark}{Remark}[section]

\newtheorem{definition}{Definition}[section]

\newcommand{\bbr}{\mathbb R}

\newcommand{\bbn} {\mathbb N}

\newcommand{\e}{\varepsilon}

\newcommand{\G}{\mathcal G}

\newcommand{\D}{\mathcal D}

\newcommand{\E}{\mathcal E}

\allowdisplaybreaks

\def\charf {\mbox{{\text 1}\kern-.30em {\text l}}}
\begin{document}
%%%%%%%%%%%%%%%%

\date{\today}

\subjclass[2010]{39A11, 39A12, 34D05, 68M10} 
\keywords{Cucker-Smale model, randomly switching topology, directed graph, flocking}

\allowdisplaybreaks

\thanks{\textbf{Acknowledgment.} The work of S.-Y. Ha is supported by the National Research Foundation of Korea(NRF-2017R1A2B2001864). The work of J. Jung is supported by the German Research Foundation (DFG) under the project number IRTG2235.}

\begin{abstract}
We present an emergent stochastic flocking dynamics of the Cucker-Smale (CS) ensemble under randomly switching topologies. The evolution of the CS ensemble with randomly switching topologies involves two random components (switching times and choices of network topologies at switching instant). First, we allow switching times for the network topology to be random so that the successive increments are i.i.d. processes following the common probability distribution. Second, at each switching instant, we  choose a network topology randomly from a finite set of admissible network topologies whose union contains a spanning tree. Even for the fixed deterministic network topology, the CS ensemble may not exhibit a mono-cluster flocking depending on the initial data and the decay mode of the communication weight functions measuring the degree of interactions between particles. 
For the flocking dynamics of the CS ensemble with these two random components, we first use a priori conditions on the network topologies and uniform boundedness of position diameter, and derive the flocking estimates via matrix theory together with a priori conditions, and then replace the a priori condition for the position diameter by some suitable condition on the system parameters and communication weight. The a priori condition on the network topology will be guaranteed by the suitable spanning tree time-blocks with probability one. 
\end{abstract}
\maketitle \centerline{\date}

%\tableofcontents

\section{Introduction} \label{sec:1}
\setcounter{equation}{0}
Collective coherent movements in many-body systems are often observed in our nature, e.g. schooling of fishes, synchronous chirps of crickets, flocking of birds, herding of sheep, etc \cite{B-C, B-H, T-T, V2}. And after seminal works on the flocking modeling by Vicsek and Reynolds \cite{Rey, V}, phenomenological models were proposed to explain such aforementioned collective behaviors from diverse scientific disciplines such as applied mathematics, biology, control theory and statistical physics, etc. Among such diverse collective behaviors, we are mainly interested in the {\it flocking} phenomena, where particles adjust their velocities to the others' based only on limited environmental information or simple rules so that all particles move with the same velocity asymptotically. In this paper, we are interested in the flocking model proposed by Cucker and Smale \cite{C-S}. To be more specific, let $x_i$ and $v_i$ be the position and velocity of the $i$-th Cucker-Smale particle in $\bbr^d$, respectively. Then, the state evolution of $(x_i, v_i)$ is governed by the following Cauchy problem:
\begin{equation} \label{CS_0}
\begin{cases}
\displaystyle \dot{x}_i = v_i, \quad t > 0, \quad 1 \le i \le N,\\
\displaystyle \dot{v}_i = \frac{1}{N}\sum_{j=1}^N \phi(\|x_j - x_i\|) \left(v_j-v_i \right), \\
\displaystyle (x_i(0), v_i(0)) = (x^{in}_{i}, v^{in}_{i}),
\end{cases}
\end{equation}
where $\|\cdot\|$ denotes the standard $\ell^2$-norm in $\bbr^d$ and the communication weight $\phi: [0,\infty) \to \bbr_+$ is bounded, Lipschitz continuous and monotonically decreasing:
\begin{align}
\begin{aligned} \label{comm}
& 0 < \phi(r) \le \phi(0)=: \kappa, \quad [\phi]_{Lip} := \sup_{x \not = y} \frac{|\phi(x) - \phi(y)|}{|x-y|} < \infty, \\
&  (\phi(r)-\phi(s))(r-s) \le 0, \quad r,s\geq 0.
\end{aligned}
\end{align}
For the model \eqref{CS_0} - \eqref{comm} with the fixed deterministic network topology, there have been a lot of extensive research activities from various points of view, i.e. emergence of flocking dynamics \cite{C-S, H-L, H-T}, effects of white noises \cite{A-H, C-M, H-L-L}, time-delay effects \cite{D-H-K2, E-H-S}, application to flight formation \cite{P-E-G}, collision avoidance \cite{C-D},  generalized network structures including hierarchical leadership \cite{C-D2, D-Q, L-H, L-X}, mean-field limit \cite{B-C-C, H-L}, kinetic and hydrodynamic description \cite{C-F-R-T, F-H-T, H-T, P-S}, extension of the CS model \cite{D-H-K, H-K-R, M-T}, etc (see a recent survey \cite{C-H-L} for details). \newline

In this paper, we consider a CS flock navigating in the free space $\bbr^d$. During their evolution, the connection topology might undergo an abrupt changes due to unknown external disturbances, obstacles and internal processing mechanisms at unknown instants. In this situation, two natural questions can arise: \newline
\begin{quote}
\begin{itemize}
\item
(Q1):~How should we model the flocking dynamics of the CS model with randomly switching network topologies?

\vspace{0.2cm} 

\item
(Q2):~If the model is properly set up, then can we find some framework leading to some kind of flocking behavior in terms of system parameters and initial data?
\end{itemize}
\end{quote}
To address the above questions, we assume that the network topologies might change along a random sequence of switching times, and at each switching time, we choose a network topology from a given finite set of admissible network topologies randomly, i.e., we employ two random components such as the random switching times and random choice of network topologies. Of course, our chosen network topology may not contain a spanning tree which is necessary for emergence of flocking. Thus, we assume that the union of network topologies in the admissible set contains a spanning tree so that on a suitable time-block with finite size, the union of network topologies contains a spanning tree. Hence, each CS particle repeatedly communicates with at least one of neighboring particles during each time-block.  With this setting in mind, we consider the evolution law for the CS flocking with randomly switching topologies similar to the model \cite{C-D3}:
\begin{align}
\begin{aligned}\label{CS}
&\dot{x}_i = v_i, \quad 1 \le i \le N, \quad t>0,\\
&\dot{v}_i = \frac{1}{N}\sum_{j=1}^N \chi_{ij}^{\sigma} \phi(\|x_j - x_i \|)  \left(v_j -v_i \right),
\end{aligned}
\end{align}
where $(\chi_{ij}^{\sigma(t)})$ denotes the time-dependent network topology corresponding to the switching law $\sigma :[0,\infty) \to \{1 , \cdots, N_G\}$. Here, we have the set of admissible (directed) graphs with $N$ vertices ${\mathcal S}:=\{\mathcal G_1,\cdots,\mathcal G_{N_G} \}$. The law $\sigma$, which is piecewise constant and right-continuous, tells which network topology is used to describe the connectivity between CS particles at a certain instant. Moreover, the sequence of discontinuities $\{t_\ell\}_{\ell \in\mathbb{N}}$ would be called the sequence of {\it switching} instants (or times). For specific description, once an instant $t$ is given, then $\sigma(t) = \sigma(t_\ell) = k$ for some $1\le k \le N_G$ and $\ell\in\mathbb{N}$, and the network topology $(\chi_{ij}^{\sigma(t)})$ corresponds to the 0-1 adjacency matrix of $k$-th digraph $\mathcal{G}_k$ (see Section \ref{sec:2.2}). However, under the influence of unpredictable disturbance, the ensemble of CS particles are subject to connection failure. To explain this {\it random failure} of connectivity, authors in previous literature \cite{D-M, D-M2, H-M, R-L-X} considered discretized analogues of the CS system and $\chi_{ij}$'s in place of $\chi_{ij}^\sigma$'s, which are assumed to be nonnegative, independent and identically distributed random variables. In our case, we focus on  the continuous system and explore this randomness in connectivity by introducing randomness into the switching law $\sigma$ and the sequence of switching instants $\{t_\ell\}_{\ell\in\mathbb{N}}$. Now, the switching law $\sigma=\sigma(t,\omega)$ $(t\geq0,~\omega\in\Omega)$ becomes a $\{1,\cdots, N_G\}$-valued jump process and the sequence $\{t_\ell\}_{\ell \in\mathbb{N}}$ has also certain randomness (see Section 2.3 for detail). To describe the random switching times $\{t_\ell \}$, we instead consider the increment process $  \{\Delta_\ell := t_{\ell +1} - t_\ell \}$ and we assume that it follows some preassigned distribution $f$ on the common probability space $(\Omega, {\mathcal F}, {\mathbb P})$. On the other hand, at each switching instant, we choose the network topology ${\mathcal G}_k$ with a probability $p_k$. \newline

Next, we briefly discuss our main result on the emergence of stochastic flocking of the model \eqref{CS}. We assume that the probability density function $f$, choice probability $p_k$ and communication weight function $\phi$ satisfy 
\[ \mbox{supp}(f) \subset [a, b], \qquad \frac{ \kappa b(N-1)}{\min\limits_{1\leq k\leq N_G}\log\frac{1}{1-p_k}}<1, \qquad \frac{1}{\phi(r)}= {\mathcal O}(r^{ \varepsilon })\quad\mbox{as}\quad r\to\infty,  \]
where $\varepsilon$ is a small positive constant. Then, under the above set of assumptions, we show that any solution process $(X,V)$ to \eqref{CS} satisfies the mono-cluster flocking with probability one (see Theorem \ref{T3.1}):
\[\mathbb{P}\Big (\omega \in \Omega:~ \exists \ x^\infty>0 \mbox{ s.t } \sup_{0\le t < \infty }\D(X(t,\omega)) \le x^\infty, \mbox{ and } \lim_{t \to \infty} \D(V(t,\omega)) =0 \Big ) = 1, \]
where $\mathcal{D}(X)$ and $\mathcal{D}(V)$ denote position and velocity diameter processes: 
\[ {\mathcal D}(X) := \max_{1 \leq i, j \leq N} \|x_i - x_j\|, \quad {\mathcal D}(V) := \max_{1 \leq i, j \leq N} \|v_i - v_j\|.
\]

\vspace{0.2cm}

The rest of this paper is organized as follows. In Section 2, we review several basic concepts on directed graphs, scrambling and state transition matrices. In Section \ref{sec:3}, we present our sufficient framework and main result for the stochastic mono-cluster flocking estimate. In Section \ref{sec:4}, we first provide a priori  flocking estimates along the sample path under two a priori assumptions on the network topologies and position diameter,  and then we replace a priori condition for the position diameter by suitable conditions on the system parameters and communication weight, and the a priori condition for the network topology will be shown to hold with probability one for a suitably chosen time-block sequence. Finally, Section \ref{sec:5} is devoted to a brief summary of our main results and some remaining issues for a future work. \\

\noindent {\bf Notation:} Throughout the paper, $(\Omega, \mathcal{F}, \mathbb{P}$) denotes a generic probability space.  Matrix ordering is meant componentwise, e.g., for matrices $A=(a_{ij})_{N\times N}$ and $B=(b_{ij})_{N\times N}$,
$A\geq B$ stands for $a_{ij}\geq b_{ij}$ for all $i$, $j$. For a real number $c$, denote by $\lfloor c\rfloor$ the floor of c, i.e., the largest integer no greater than $c$. $\mathbb{N}$ denotes the set of all natural numbers (including zero). 
\section{Preliminaries}\label{sec:2}
\setcounter{equation}{0}
In this section, we first study the dissipative structure of system \eqref{CS}, and briefly review several notions on the directed graphs, scrambling matrices and state transition matrices.

\subsection{Pathwise dissipative structure} \label{sec:2.1} In this subsection, we study the dissipative structure of system \eqref{CS} with randomly switching topologies. For the symmetric network topology, the R.H.S. of $\eqref{CS}_2$ is skew-symmetric under the exchange symmetry $i \longleftrightarrow j$. Hence, the total momentum  $\sum_{i=1}^{N} v_i$ is a constant of motion. In contrast, for a digraph topology, the R.H.S. of system \eqref{CS}  may not be skew-symmetric under the exchange symmetry. This breaks up the conservation law for the total momentum.  Despite of this, we can still see that the velocity diameter is non-increasing pathwise. 
\begin{lemma} \label{L2.1}
Let $(X,V)$ be a solution process to \eqref{CS}. Then, the velocity diameter ${\mathcal D}(V)$ is non-increasing pathwise: for each $\omega \in \Omega$, 
\[
\frac{d}{dt}\D(V(t,\omega))\leq0,\quad \mbox{a.e. }t>0.
\]
\end{lemma}
\begin{proof}
For a given $t \geq 0$ and $\omega \in \Omega$, let $i$ and $j$ be indices satisfying the relation:
\begin{equation} \label{C-0}
 \D(V(t,\omega))=\|v_i(t, \omega)-v_j(t, \omega) \|.
\end{equation} 
In the sequel, for a notational simplicity, we suppress $t$ and $\omega$ dependence in $v_i$:
\[ v_i = v_i(t, \omega). \]
Then, for such $i$ and $j$, we have
\begin{align}
\begin{aligned}\label{C-13-1}
&\frac{1}{2}\frac{d}{dt}\|v_i-v_j\|^2 =\left\langle v_i-v_j,\frac{dv_i}{dt}-\frac{dv_j}{dt}\right\rangle\\
& \hspace{0.5cm} =\left\langle v_i-v_j,\frac{1}{N}\sum_{k=1}^N \chi^{\sigma}_{ik}\phi_{ik}( v_k- v_i)\right\rangle +\left\langle v_j-v_i,\frac{1}{N}\sum_{k=1}^N \chi^{\sigma}_{jk}\phi_{jk}( v_k- v_j)\right\rangle\\
& \hspace{0.5cm} =:\mathcal{I}_{11}+\mathcal{I}_{12},
\end{aligned}
\end{align}
where $\langle \cdot, \cdot \rangle$ denotes the standard inner product in $\bbr^d$, and we wrote 
\[ \phi_{ij}:=\phi(\|x_i-x_j\|), \quad  i,j=1,2,\cdots, N \]
for convenience. Below, we estimate the terms ${\mathcal I}_{1i},~i=1,2$ one by one. \newline

\noindent $\bullet$ (Estimate of $\mathcal{I}_{11}$): For $k = 1, \cdots, N$, we use the relation \eqref{C-0} to get 
\begin{equation} \label{C-13-2}
\langle v_k-v_i,v_i-v_j\rangle =\frac{\|v_k-v_j\|^2-\|v_k-v_i\|^2-\|v_i-v_j\|^2}{2} \leq \frac{\|v_i-v_j\|^2-0-\|v_i-v_j\|^2}{2} =0.
\end{equation}
This yields
\begin{equation} \label{C-13-3}
\mathcal{I}_{11} =  \frac{1}{N}\sum_{k=1}^N\chi^{\sigma}_{ik} \phi_{ik}\left\langle v_i-v_j,v_k- v_i\right\rangle \leq   0. 
\end{equation}

\vspace{0.2cm}

\noindent $\bullet$ (Estimate of $\mathcal{I}_{12}$) : Similar to \eqref{C-13-2}, we also have
\[ 
\langle v_k-v_j,v_j-v_i\rangle =\frac{\|v_k-v_i\|^2-\|v_k-v_j\|^2-\|v_j-v_i\|^2}{2} \leq \frac{\|v_j-v_i\|^2-0-\|v_j-v_i\|^2}{2}=0. \]
This again implies
\begin{equation} \label{C-13-5}
\mathcal{I}_{12} =   \frac{1}{N}\sum_{k=1}^N\chi^{\sigma}_{jk} \phi_{jk}\left\langle v_j-v_i, v_k- v_j\right\rangle  \leq  0.
\end{equation}
In \eqref{C-13-1}, we use $\D(V)=\|v_i-v_j\|$ and combine estimates \eqref{C-13-3} and \eqref{C-13-5} to get 
\[
\D(V(t))\frac{d}{dt}\D(V(t))\leq0,\quad\mbox{a.e. }t>0.
\]
If $\D(V(t))>0$, then we can divide the above inequality by $\D(V(t))$ to obtain the desired estimate. \newline

On the other hand, if $\D(V(t))=0$ and diffrentiable at $t$, then $\D(V)$ attains a global minimum at $t$, so $\frac{d}{dt}\D(V(t))=0$. Hence we have the following differential inequality:
\[
\frac{d}{dt}\D(V(t))\leq0,\quad \mbox{a.e. }t>0.
\]
\end{proof}
\begin{remark}
Note that the result of Lemma \ref{L2.1} illustrates that the velocity diameter is non-increasing in time. Now, our job is to find some conditions leading to the zero convergence of velocity diameter. This will be done in Section \ref{sec:4}.
\end{remark}

\subsection{A directed graph} \label{sec:2.2}
In this subsection, we review jargons for network topology modeling by a directed graph (digraph). A digraph $\mathcal{G} = (\mathcal{V}({\mathcal G}), \mathcal{E}({\mathcal G}))$ consists of two sets: a set of vertices (nodes) $\mathcal{V}({\mathcal G}) = \{1, \cdots, N\}$ with $|{\mathcal G}| = N$, and a set of edges $\mathcal{E}({\mathcal G}) \subset \mathcal{V} \times \mathcal{V}$ consisting of ordered pairs of vertices:
\begin{align*}
\begin{aligned}
 (j, i) \in {\mathcal E}({\mathcal G}) \quad &\iff \quad \mbox{vertex $i$ receives an information (or signal) from the vertex $j$} \\
 &\iff \quad \mbox{$j$ is a neighbor of $i$}.
 \end{aligned}
 \end{align*}
 In this case, we define a neighbor set ${\mathcal N}_i$ of the vertex $i$:
 \[ {\mathcal N}_i :=  \{ j  \in \mathcal{V}({\mathcal G}):~ (j,i)\in \mathcal{E}({\mathcal G}) \}. \]
 If $(i,i)\in \mathcal{E}({\mathcal G})$, then we say that $\mathcal{G}$  has a self-loop at $i$. If $\mathcal{G}$ does not have a self-loop at any vertices, then $\mathcal{G}$ is said to be {\it simple}. \newline
 
 For a given digraph $\mathcal{G} = (\mathcal{V}({\mathcal G}), \mathcal{E}({\mathcal G}))$, we consider its $(0,1)$-adjacency matrix $\chi = (\chi_{ij})$:
\[\chi_{ij} := \left\{\begin{array}{ccc}
1 & \mbox{ if } & (j,i)\in \mathcal{E}({\mathcal G}),\\
0 & \mbox{ if } & (j,i)\notin \mathcal{E}({\mathcal G}).
\end{array}\right.\]
A path in $\mathcal{G}$ from $i$ to $j$ is a sequence of ordered distinct vertices  $(i_0 = i, \cdots, i_n = j)$:
\[ i= i_0 ~\longrightarrow~i_1~\longrightarrow \quad \cdots \quad \longrightarrow i_n = j  \quad \mbox{such that}~~ (i_{m-1}, i_m) \in \mathcal{E}({\mathcal G}) \quad \mbox{for every $1 \le m \le n$}. \]
If there is a path from $i$ to $j$, then we say $j$ is reachable from $i$. Moreover, a digraph $\mathcal{G}$ is said to have a spanning tree if $\mathcal{G}$ has a vertex $i$ from which any other vertices are reachable. As long as there is no confusion, we suppress ${\mathcal G}$-dependence in  $\mathcal{G} = (\mathcal{V}({\mathcal G}), \mathcal{E}({\mathcal G}))$ throughout the paper:
\[ \mathcal{V} = \mathcal{V}({\mathcal G}), \quad   \mathcal{E} = \mathcal{E}({\mathcal G}). \]
\subsection{A scrambling matrix}\label{sec:2.3}
Next, we recall the concept of scrambling matrices. First, we introduce several concepts of nonnegative matrices in the following definition.
\begin{definition} \label{D2.1}
Let $A=(a_{ij})$ be a nonnegative $N \times N$ matrix whose entries are nonnegative.
\begin{enumerate}
\item
$A$ is a {\em stochastic} matrix, if its row-sum is equal to unity:
\[ \sum_{j=1}^{N} a_{ij} = 1, \quad 1 \leq i \leq N. \]
\item
$A$ is a {\em scrambling} matrix, if for each pair of indices $i$ and $j$, there exist an index $k$ such that 
\[ a_{ik}>0 \quad \mbox{and} \quad a_{jk}>0. \]
\item
$A$ is an {\em adjacency matrix} of a digraph $\G$ if the following holds:
\[ a_{ij}>0 \quad \Longleftrightarrow \quad (j,i)\in\E. \]
In this case, we write $\G=\G(A)$.

\end{enumerate}
\end{definition}
\begin{remark}
Define the {\em ergodicity coefficient} of $A$ as follows.
\begin{equation} \label{B-1}
\mu(A):=\min_{i,j}\sum_{k=1}^N \min\{a_{ik},a_{jk}\}.
\end{equation}
Then, it is easy to see that
\begin{enumerate}
\item
$A$ is scrambling if and only if $\mu(A)>0$.
\item
For nonnegative matrices $A$ and $B$,
\begin{equation} \label{B-1-1}
 A\geq B \quad \Longrightarrow \quad \mu(A)\geq\mu(B).
\end{equation}
\end{enumerate}
\end{remark}

\vspace{0.2cm}

For a $N \times N$ matrix $A = (a_{ij})$, the Frobenius norm of $A$ is defined as follows.
\[ \| A \|_F := \sqrt{\mbox{trace}(A A^*)} = \sqrt{\mbox{trace}(A^* A)}. \]
In the following lemma, we state some properties of scrambling matrices without proofs. 
\begin{lemma}\label{L2.1}
\emph{(Lemma 2.2, \cite{D-H-K})}
Suppose that a nonnegative $N\times N$ matrix $A=(a_{ij})$ is stochastic, and let $B=(b_{i}^{j})$, $Z=(z_{i}^{j})$ and $W=(w_{i}^{j})$ be $N\times d$ matrices such that
\begin{equation*} %\label{B-2}
 W = AZ+B.
\end{equation*}
Then, we have
\begin{equation*}
\max_{i,k}\|w_i-w_k\|\leq(1-\mu(A))\max_{l,m}\|z_l-z_m\|+\sqrt{2}\|B\|_F,
\end{equation*}
where
\[
z_i:=(z_{i}^{1},\cdots,z_{i}^{d}), \quad  b_i:=(b_{i}^{1},\cdots,b_{i}^{d}),  \quad w_i:=(w_{i}^{1},\cdots,w_{i}^{d}), \quad i=1,\cdots,N.
\]
\end{lemma}
%
%\begin{lemma}\label{L2.2}
%\emph{(Lemma 2, \cite{JLM03})}
%Let $m\geq2$ be a positive integer and let $A_1,A_2,\ldots,A_m$ be
%nonnegative $N\times N$ matrices. Suppose that the diagonal elements
%of all of the $A_i$ are positive and let $\underline{\lambda}$ and
%$\overline{\lambda}$ be the smallest and largest of these,
%respectively. Then
%\begin{equation*}
%A_1A_2\cdots
%A_m\geq\left(\frac{\underline{\lambda}^2}
%{2\overline{\lambda}}\right)^{(m-1)}
%(A_1+A_2+\cdots+A_m).
%\end{equation*}
%\end{lemma}

\begin{proposition}\label{P2.1}
\emph{(Theorem 5.1, \cite{Wu06})}
Let $A_i$ be nonnegative $N\times N$ matrices with positive diagonal
elements such that $\G(A_i)$ has a spanning tree for all $1\leq i\leq N-1$. Then, one has
\[ A_1A_2\ldots A_{N-1}~\mbox{is a scrambling matrix}. \]
\end{proposition}

\subsection{A state transition matrix}\label{sec:2.4}
In this subsection, we discuss the notion and properties of state transition matrices. Let $t_0\in\bbr$ and $A:[t_0,\infty)\to\bbr^{N\times N}$ be an $N\times N$ matrix of piecewise continuous function.  \newline

Consider the following Cauchy problem for the time-dependent linear ODE:
\begin{align}
\begin{aligned} \label{B-4}
& \frac{d\xi(t)}{dt} =A(t)\xi(t),\quad t > t_0, \\
& \xi |_{t = t_0} = \xi(t_0).
\end{aligned}
\end{align}
Then, the solution of \eqref{B-4} is given by
\begin{equation*}
\xi(t)=\Phi(t,t_0)\xi(t_0), \quad t \geq t_0,
\end{equation*}
where $\Phi(t,t_0)$ is called the state transition matrix or the fundamental matrix for \eqref{B-4}. \newline

Note that we can write the state transition matrix $\Phi(t,t_0)$ corresponding to system \eqref{B-4} as the Peano-Baker series (see \cite{Sontag}):
\[
\Phi(t,t_0)=I+\sum_{n=1}^\infty \int_{t_0}^t\int_{t_0}^{\tau_1}\cdots\int_{t_0}^{\tau_{n-1}}A(\tau_1)A(\tau_{2})\cdots A(\tau_n)d\tau_n\cdots d\tau_2d\tau_1,
\]
where $I$ is the  $N \times N$ identity matrix. \newline

Let $t_0\in\bbr$, $c\in\bbr$ and $A:[t_0,\infty)\to\bbr^{N\times N}$ be an $N\times N$ matrix of continuous functions. Then, for such time-dependent matrix $A$, we set $\Phi(t,t_0)$ and $\Psi(t,t_0)$ to be the state transition matrices corresponding to the following linear ODEs, respectively:
\begin{equation*} %\label{B-4-0}
\frac{d\xi(t)}{dt} =A(t)\xi(t) \quad \mbox{and} \quad
\frac{d\xi(t)}{dt} =[A(t)+cI]\xi(t),\quad t > t_0.
\end{equation*}
In the next lemma, we study a relation between $\Phi(t,t_0)$ and $\Psi(t,t_0)$ to be used in Lemma \ref{L4.1}.
\begin{lemma}\label{L2.2}
\emph{\cite{D-H-K}}
The following relation holds.
\[ \Phi(t,t_0)=e^{-c(t-t_0)}\Psi(t,t_0), \quad \mbox{or} \quad \Psi(t,t_0) =e^{c(t-t_0)}  \Phi(t,t_0), \quad t \geq t_0. \]
\end{lemma}
\begin{proof}
The proof can be found in Lemma 2.3 of \cite{D-H-K}.
\end{proof}

%
%
%\subsection{Basic estimates} In this subsection, we provide basic estimates for system \eqref{TCS}.
%
%\begin{lemma}\label{L2.1}
%Let $(X,V,\Theta)$ be a solution to system \eqref{TCS}. Then, one has
%\[\mathbb{P}( \omega : \mathcal{D}(\Theta(t,\omega)) \mbox{ is monotonically decreasing }) =1. \]
%\end{lemma}
%\begin{proof}
%Our proof is based on the analysis on each sample path. We fix $\omega \in \Omega$ and we set extremal indices as
%\[\theta_M(t, \omega) := \max_{1 \le i \le N} \theta_i(t,\omega), \quad \theta_m(t,\omega):= \min_{1 \le i \le N}\theta_i(t,\omega). \]
%Then for a.e. $t \ge 0$, one gets
%
%\begin{align*}
%&\frac{d\theta_M}{dt}(t,\omega) = \frac{1}{N}\sum_{j=1}^N \zeta(|x_j - x_M|) \chi_{ij}^{\sigma(t,\omega)}\left( \frac{1}{\theta_M} - \frac{1}{\theta_j}\right) \le 0,\\
%&\frac{d\theta_m}{dt}(t,\omega) = \frac{1}{N}\sum_{j=1}^N \zeta(|x_j - x_m|) \chi_{ij}^{\sigma(t,\omega)}\left( \frac{1}{\theta_m} - \frac{1}{\theta_j}\right) \ge 0,\\
%\end{align*}
%which implies that $\mathcal{D}(\Theta(t,\omega))$ is monotonically decreasing for the fixed $\omega \in \Omega$. However, since the choice of $\omega \in \Omega$ was arbitrary, we can obtain our desired result.
%
%\end{proof}

\section{A description of main result} \label{sec:3}
\setcounter{equation}{0}
In this section, we present a  framework and main result for the emergence of stochastic flocking to the CS model with randomly switching topologies.
\subsection{Standing assumptions} \label{sec:3.1}
Let $\{t_\ell\}_{\ell\in \mathbb{N}}$ be an increasing sequence of ``random switching times" such that the increment sequence  $\{t_{\ell+1}-t_\ell \}_{\ell\in\mathbb{N}}$ is a sequence of i.i.d. positive random variables on the common probability space $(\Omega, \mathcal{F}, \mathbb{P}$) with probability density function $f$. We also assume that the switching law $\{\sigma _t\}_{t\geq0}$ satisfies the following conditions:
\begin{itemize}
\item For each $\ell \geq 0$ and $\omega \in \Omega$, $\sigma_t(\omega) = \sigma(t,\omega)$ is constant on the interval $t\in[t_\ell(\omega),t_{\ell+1}(\omega))$.
\item $\{\sigma_{t_{\ell}} \}_{\ell\ge 0}$ is a sequence of i.i.d. random variables such that for any $\ell \geq 0$,
\[
\mathbb P(\sigma_{t_\ell}= k)=p_k,\quad \mbox{ for each } k=1,\cdots, N_G,
\]
where $p_1,\cdots,p_{N_G} >0$ are given positive constants satisfying $p_1+\cdots+p_{N_G} =1$.
\end{itemize}
For each $k=1, \cdots, N_G$, let $\mathcal G_k=(\mathcal V, \mathcal E_k)$ be the $k$-th admissible digraph, and for each $t\geq0$ and $\omega\in\Omega$, the time-dependent network topology $(\chi_{ij}^\sigma)=(\chi_{ij}^{\sigma_t(\omega)})$ is determined by

\[\chi_{ij}^{\sigma_t(\omega)} := \left\{\begin{array}{ccc}
1 & \mbox{ if } & (j,i)\in\mathcal E_{\sigma_t(\omega)},\\
0 & \mbox{ if } & (j,i)\notin \mathcal E_{\sigma_t(\omega)}.
\end{array}\right.\]
\vspace{0.1cm}

\noindent For technical reasons and without loss of generality, we assume that each $\mathcal{G}_k$ has a self-loop at each vertex. For later use, we define the union graph of $\mathcal{G}_{\sigma_t(\omega)}$ for $t \in [s_0, s_1)$ and $\omega\in\Omega$ as
\[\mathcal{G}([s_0, s_1))(\omega) := \bigcup_{t \in [s_0, s_1)}\mathcal{G}^{\sigma_t(\omega)} =\left(\mathcal{V}, \bigcup_{t \in [s_0, s_1)}\mathcal{E}_{\sigma_t(\omega)}\right).\\ \]

Note that the network topology might not actually `switch' at the (possibly) switching instants. In other words, it might happen that $\sigma_{t_{\ell+1}}(\omega)=\sigma_{t_{\ell}}(\omega)$ for some $\ell\geq0$ and $\omega \in \Omega$. Now, we are ready to provide a framework for stochastic flocking to the random dynamical system \eqref{CS}.  \newline

For a set of admissible digraphs and the probability density fuction $f$ of increments of switching times, we impose the following assumption $(\mathcal A)$ as our standing assumption throughout the paper.
\begin{itemize}
\item $(\mathcal A1)$: The union digraph of all available network topologies in the set ${\mathcal S}$ has a spanning tree:
\[
\bigcup_{1\leq k\leq n}\mathcal{G}_k :=\left(\mathcal{V}, \bigcup_{1\leq k\leq n}\mathcal{E}_{k}\right)\quad\mbox{has a spanning tree.}
\]
\item $(\mathcal A2)$: $f$ is supported on some bounded interval with a positive lower bound, say
\[
\mbox{supp}f\subset[a,b]\subset(0,\infty).
\]
\end{itemize}

\subsection{Main result} \label{sec:3.2} Below, we first briefly sketch our proof strategy and then present our main result. Basically, we will use matrix theory discussed in the previous section as key tools for the flocking estimate along sample paths. More precisely, we delineate our proof strategy in four steps. \newline

\begin{itemize}
\item
Step A~(Matrix formulation): In order to use matrix theory,  we rewrite the momentum equation $\eqref{CS}_2$ as a matrix form:
\[  \frac{d}{dt}V(t)=-\frac{1}{N}  L_{\sigma_t }(t)V(t),   \]
where $L_{\sigma_t }(t)$ is the Laplacian matrix to be defined in \eqref{D-1-1} - \eqref{D-1-2}.

\vspace{0.2cm}

\item
Step B~(A priori velocity alignment estimate along a sample path):~For each sample point $\omega \in \Omega$, we introduce a priori conditions:
\begin{enumerate}
\item
(${\mathcal P}$1):~there exist $n\in\bbn$ and  $c>0$  such that $\kappa b(N-1)c<1$, and the subsequence $\{t_\ell^*\}_{\ell\in\bbn}\subset\{t_\ell\}_{\ell\in\bbn}$ defined by $t_\ell^*:= t_{a_\ell(n,c)}$ satisfies
\[
 \mathcal G([t_\ell^*,t_{\ell+1}^*))(\omega)~\mbox{has a spanning tree for all $\ell\geq0$}, 
\]
where the explicit construction of $a_\ell(n,c)$ will be given in \eqref{D-3-1}.

\vspace{0.2cm}

\item
(${\mathcal P}$2):~the position diameter is uniformly bounded pathwise:
\[
\sup_{0\leq t<\infty}\D(X(t,\omega))\leq x^{\infty} < \infty.
\]
Note that the constant $x^{\infty}$ can be chosen independent of $\omega$ in this step. 
\end{enumerate}
Under the above two a priori assumptions, we show that the velocity alignment estimate can emerge (Proposition \ref{P4.1}):
\[ \lim_{t \to \infty} \D(V(t, \omega)) = 0. \]

\vspace{0.2cm}

\item
Step C (Flocking along a sample path):~We replace the a priori assumption (${\mathcal P}2)$ by a suitable condition on the system parameters and communication weight, and derive flocking estimates along sample path: for each $\omega \in \Omega$, 
\[ \sup_{0\leq t<\infty}\D(X(t,\omega))\leq x^{\infty} < \infty, \quad  \lim_{t \to \infty} \D(V(t, \omega)) = 0. \]

\vspace{0.2cm}

\item
Step D  (Stochastic flocking):~We look for a suitable condition for the choice probability $p_k$ for the network selection, and construct a suitable time-block guaranteeing an existence of spanning tree in each time-block, and then under these well-prepared setting, the a priori assumption  (${\mathcal P}1)$ can be attained with probability one.
\end{itemize}

\vspace{0.2cm}

We perform the above outlined strategy one by one to derive our main result on the flocking estimate of \eqref{CS} .
\begin{theorem}\label{T3.1}
Suppose that the framework $(\mathcal A1)-(\mathcal A2)$ holds, and system parameters $b$, $N$, $p_k$'s and communication weight $\phi$ satisfy the following conditions:
\[  \frac{ \kappa b(N-1)}{\min\limits_{1\leq k\leq N_G}\log\frac{1}{1-p_k}}<1 \qquad \mbox{and} \qquad \frac{1}{\phi(|x|)}= {\mathcal O}(|x|^{ \varepsilon })\quad\mbox{as}\quad |x| \to\infty,  \]
where $\varepsilon$ is a positive constant satisfying the following relation:
\[  0 \leq\varepsilon<\frac{1}{N-1}-\frac{\kappa b }{\min\limits_{1\leq k\leq N_G}\log\frac{1}{1-p_k}}. \]
 Then, for any solution process $(X,V)$ to \eqref{CS}, the asymptotic flocking emerges with probability one:
\[\mathbb{P}\Big( \omega \in \Omega:~ \exists \ x^\infty>0 \mbox{ s.t } \sup_{0\le t < \infty }\D(X(t,\omega)) \le x^\infty, \mbox{ and } \lim_{t \to \infty} \D(V(t,\omega)) =0 \Big) = 1. \]
\end{theorem}

\section{Emergent behavior of the randomly switching system}\label{sec:4}
\setcounter{equation}{0}
In this section, we present a proof for Theorem \ref{T3.1} following the outline depicted in Section \ref{sec:3.2}.

%
%\begin{remark}
%In the proof of Theorem \ref{T3.1}, the choice of $a_{l+1} - a_l$ does not need to be a linear function on $l$. For example, if we let
%\[a_{l+1} - a_l := [l^\nu] + d, \quad \nu >0, \]
%then we have $n = [l^\nu]$ if and only if $n^{1/\nu} \le l <(n+1)^{1/\nu}$. Since the number of natural numbers in an interval $[x,y)$ is $[y]-[x]$, we have
%
%\begin{align*}
%&\sum_{l=0}^\infty (1-r)^{a_{l+1}-a_l}\\
%&= \sum_{l=0}^\infty ([(l+1)^{1/\nu}]-[l^{1/\nu}]) (1-r)^{l+d}\\
%&\le (1-r)^d \left( 1+ \sum_{l=1}^\infty \left( (l+1)^{1/\nu} - (l-1)^{1/\nu}\right) (1-r)^l \right),
%\end{align*}
%and we can use the root test to show that the series is convergent. Thus, we can choose a sufficiently large $d$ so that  the inequality in Theorem 3.1 can hold for given $\delta \in (0,1)$. Furthermore, logarithmic function can also be used under some restriction: we let
%\[a_{l+1} - a_l := [\log_c l] + d, \quad l \ge 1, \quad c \in (1,\infty), \quad a_1 - a_0 = d. \]
%Similar to the previous case, one has
%
%\begin{align*}
%&\sum_{l=0}^\infty (1-r)^{a_{l+1}-a_l}\\
%&= (1-r)^d \left(1+\sum_{l=0}^\infty \left([c^{l+1}]-[c^l]\right) (1-r)^l\right)\\
%&\le (1-r)^d \left(1+ \sum_{l=0}^\infty \left(c^{l+1} - c^{l-1}\right) (1-r)^l \right)\\
%&= (1-r)^d \left(1+ (c-c^{-1})\sum_{l=0}^\infty (c(1-r))^l \right),
%\end{align*}
%and for the series to be convergent, $c<1/(1-r)$ is required. Thus, once we have $c<1/(1-r)$ and choose a sufficiently large $d$, we can still obtain our desired result.
%
%\end{remark}

\subsection{A matrix formulation} \label{sec:4.1} In this subsection, we first reformulate the momentum equations in $\eqref{CS}_2$ so  that we can use tools from matrix theory documented in Section \ref{sec:2}.  \newline

Consider the momentum equations:
\begin{equation} \label{D-0}
 \dot{v}_i = \frac{1}{N}\sum_{j=1}^N \chi_{ij}^{\sigma} \phi(\|x_j- x_i \|)  \left(v_j -v_i \right), \quad 1 \leq i \leq N.
\end{equation} 
We rearrange the terms in  $\eqref{D-0}$ as follows.
\begin{equation} \label{D-1}
{\dot v}_i = -\frac{1}{N} \Big[ \Big( \sum_{j=1}^N \chi_{ij}^{\sigma}\phi(\|x_i-x_j\|)  \Big) v_i -\sum_{j=1}^N \chi_{ij}^{\sigma}\phi(\|x_i-x_j\|) v_j  \Big].
\end{equation}
For the matrix formulation of \eqref{D-1}, we introduce  $N\times N$ Laplacian matrices $ {L_k}(t)$ $(k=1,\cdots, N_G)$ as follows:
\begin{equation} \label{D-1-1}
  L_{k }(t):=  D_{k }(t)-  A_{k }(t),
\end{equation}
where $A_{k }(t)=\left(  a_{ij}^{k}(t)\right)$ and $D_{k }(t)=\mbox{diag}\left(  d_1^{k}(t),\cdots, d_N^{k }(t)\right)$ are written as
\begin{equation} \label{D-1-2}
  a_{ij}^{k }(t):=\chi_{ij}^{k}\phi(\|x_i(t)-x_j(t)\|)\quad\mbox{and}\quad   d_i^{k }(t)=\sum_{j=1}^N \chi_{ij}^{k}\phi(\|x_i(t)-x_j(t)\|).
\end{equation}
Thus, system $\eqref{D-1}$ can be rewritten as 
\begin{equation}\label{D-2}
\frac{d}{dt}V(t)=-\frac{1}{N}  L_{\sigma_t }(t)V(t).
\end{equation}
Let $\Phi(t_2,t_1)$ be the state transition matrix associated with \eqref{D-2} on the interval $[t_1, t_2]$. Then we have the representation formula for $V$:
\begin{equation} \label{D-3}
 {V}(t_2) = \Phi (t_2, t_1 )V(t_1), \quad  t_2\geq t_1\geq0.
\end{equation}

\subsection{Pathwise flocking under a priori assumptions} \label{sec:4.2} In this subsection, we study the emergence of stochastic flocking estimate under a priori assumption on the uniform bound for position diameter. From now on, we present {\it a priori} flocking estimates for each fixed sample $\omega \in \Omega$. In the sequel, as long as there is no confusion, we frequently suppress the $\omega$-dependence of solution processes or parameters for convenience. \newline

\noindent {\bf A priori assumptions}: For each positive integer $n$ and positive real number $c>0$, we define an increasing sequence $\{a_\ell(n,c)\}_{\ell\in\mathbb{N}}$ of integers by the following recurrence relation:
\begin{equation} \label{D-3-1}
a_0(n,c)=0,\qquad a_{\ell+1}(n,c)=a_\ell(n,c)+n+\lfloor c\log (\ell+1)\rfloor, \qquad (\ell\in\mathbb{N}).
\end{equation}
Let $\omega\in\Omega$ be fixed, and let $(X,V)$ be a solution process to \eqref{CS}. Then, our two a priori assumptions are as follows: \newline
\begin{itemize}
\item
(${\mathcal P}$1):~there exist $n\in\bbn$, $n>0$ and  $c>0$  such that $\kappa b(N-1)c<1$, and the subsequence $\{t_\ell^*\}_{\ell\in\bbn}\subset\{t_\ell\}_{\ell\in\bbn}$ defined by $t_\ell^*:= t_{a_\ell(n,c)}$ in \eqref{D-3-1} satisfies
\begin{equation*} \label{D-4}
 \mathcal G([t_\ell^*,t_{\ell+1}^*))(\omega)~\mbox{has a spanning tree for all $\ell\geq0$}, 
\end{equation*} 

\vspace{0.5cm}

\item
(${\mathcal P}$2):~the position diameter is uniformly bounded in time:
\begin{equation*} \label{D-5}
\sup_{0\leq t<\infty}\D(X(t,\omega))\leq x^{\infty} < \infty.
\end{equation*}
\end{itemize}

\vspace{0.5cm}

\begin{lemma} \label{L4.1}
Suppose that the a priori assumptions (${\mathcal P}$1) and (${\mathcal P}$2) hold. Then, the transition matrix $ \Phi(t^*_{r(N-1)}, t^*_{(r-1)(N-1)})$ is stochastic and its ergodicity coefficient satisfies
\begin{equation} \label{D-6}
  \mu\left(\Phi(t^*_{r(N-1)}, t^*_{(r-1)(N-1)})\right)\geq  e^{-\kappa (t^*_{r(N-1)}- t^*_{(r-1)(N-1)})} \left(  \frac{a}{N}\right)^{N-1}
\phi(x^\infty)^{N-1}.
 \end{equation}
\end{lemma}
\begin{proof} First, we focus on the second assertion, and we claim:
\begin{align}
\begin{aligned}\label{E-4}
\Phi(t^*_{r(N-1)}, t^*_{(r-1)(N-1)}) &\geq e^{-\kappa (t^*_{r(N-1)}- t^*_{(r-1)(N-1)})}
\left(  \frac{a}{N}\right)^{N-1}
\phi(x^\infty)^{N-1}\prod_{i=1}^{N-1}F_i,
%&\geq C_1 \zeta(x^\infty)^{N-1}\prod_{i=1}^{N-1}F_i
\end{aligned}
\end{align}
 where, for each $i=1,\ldots,N-1$, $F_i$ is the $(0,1)$-adjacency matrix of the union digraph
$$\G([t^*_{(r-1)(N-1)+i-1},t^*_{(r-1)(N-1)+i})).$$
{\it Proof of claim \eqref{E-4}}:  Let $\{t_{\ell_1},t_{\ell_2},\ldots,t_{\ell_{q+1}}\}$ be the subsequence
of $\{t_\ell\}_{\ell\geq0}$ contained in the interval
$[t^*_{(r-1)(N-1)+i-1},t^*_{(r-1)(N-1)+i}]$ such that 
\[ t_{\ell_1}=t^*_{(r-1)(N-1)+i-1} \quad \mbox{and} \quad t_{\ell_{q+1}}=t^*_{(r-1)(N-1)+i}. \]
We set
\[ \sigma_t=k_p \quad \mbox{for~~$t\in[t_{\ell_p},t_{\ell_{p+1}})$ and $p=1,\ldots,q$}. \]
Then we have\newline

\vspace{-0.6cm}

\begin{equation}\label{E-4-1}
\Phi(t^*_{(r-1)(N-1)+i}, t^*_{(r-1)(N-1)+i-1})=\Phi_{k_q}(t_{\ell_{q+1}}, t_{\ell_q})\cdots\Phi_{k_1}(t_{\ell_2}, t_{\ell_1}),
\end{equation}
where, for $p=1,\ldots, q$, $\Phi_{k_p}(t_{\ell_{p+1}}, t_{\ell_p})$ is the state transition matrix corresponding to system \eqref{D-3} on $[t_{\ell_p},t_{\ell_{p+1}})$. We need to estimate $\Phi_{k_p}(t_{\ell_{p+1}}, t_{\ell_p})$ and for this, we estimate the coefficient matrix for \eqref{D-2} as follows:
\begin{equation} \label{E-5}
-\frac{1}{N} L_{k_p}(t)
=\frac{1}{N} (A_{k_p}(t)-D_{k_p}(t))
\geq\frac{1}{N}\underline{A}_{k_p} -\kappa I,
\end{equation}
where $\underline{A}_{k_p}=(\underline{a}_{ij}^{k_p})$ is given by
\begin{equation*}
\underline{a}_{ij}^{k_p}:=\begin{cases}
\chi_{ij}^{k_p}\phi(x^\infty), \quad & i\neq j, \cr
\kappa, \quad & i=j.
\end{cases}
\end{equation*}
Then, the relation \eqref{E-5} implies 
\begin{equation}\label{E-6}
-\frac{1}{N} L_{k_p}(t)+\kappa I\geq \frac{1}{N}\underline{A}_{k_p
}\geq0.
\end{equation}
On the other hand, let $\Psi_{k_p}(t_{\ell_{p+1}}, t_{\ell_p})$ be the state transition matrix of 
\[ -\frac{1}{N} L_{k_p}(t)+\kappa I \quad \mbox{on}~~[t_{\ell_p}, t_{\ell_{p+1}}). \]
Then it follows from Lemma \ref{L2.2} that
\begin{equation} \label{E-7}
\Phi_{k_p}\left(t_{\ell_{p+1}},t_{\ell_p}\right)=
e^{-\kappa (t_{\ell_{p+1}}-t_{\ell_p})}
\Psi_{k_p}\left(t_{\ell_{p+1}},t_{\ell_p}\right).
\end{equation}
Now, we can apply \eqref{E-6} to the Peano-Baker series to obtain

\begin{align} \label{E-8}
\begin{aligned}
&\Psi_{k_p}\left(t_{\ell_{p+1}},t_{\ell_p}\right) \\
& \hspace{0.5cm} =I+\sum_{n=1}^\infty \hspace{-0.05cm}\int_{t_{\ell_p}}^{t_{\ell_{p+1}}} \hspace{-0.25cm} \int_{t_{\ell_p}}^{\tau_1}\hspace{-0.2cm}\cdots\hspace{-0.1cm}\int_{t_{\ell_p}}^{\tau_{n-1}}
\Big[\hspace{-0.05cm}\big(-\frac{1}{N} L_{k_p}(\tau_1)+\kappa I \big)\cdots \big(-\frac{1}{N} L_{k_p}(\tau_n) +\kappa I\big)\hspace{-0.05cm} \Big]d\tau_n\cdots d\tau_1\\
&\hspace{0.5cm} \geq I+\sum_{n=1}^\infty\hspace{-0.05cm}\int_{t_{\ell_p}}^{t_{\ell_{p+1}}} \hspace{-0.25cm} \int_{t_{\ell_p}}^{\tau_1}\hspace{-0.2cm}\cdots\hspace{-0.1cm}\int_{t_{\ell_p}}^{\tau_{n-1}}\Big(\frac{1}{N}\underline{A}_{k_p} \Big)^n d\tau_n\cdots d\tau_1\\
&\hspace{0.5cm} = I+\sum_{n=1}^\infty\frac{1}{n!} (t_{\ell_{p+1}}-t_{\ell_p})^n\Big(\frac{1}{N}\underline{A}_{k_p}\Big)^n \\
&\hspace{0.5cm} \geq I+ \frac{a}{N}\underline{A}_{k_p}.
\end{aligned}
\end{align}
We combine \eqref{E-7} with \eqref{E-8} to obtain

\begin{equation} \label{E-8-1}
\Phi_{k_r}\left(t_{\ell_{p+1}},t_{\ell_p}\right)\geq
e^{-\kappa (t_{\ell_{p+1}}-t_{\ell_p})}\left(I+
  \frac{a}{N}
\underline{A}_{k_p}\right).
\end{equation}
Then, the relation \eqref{E-8-1} and \eqref{E-4-1} yield
\begin{align}
\begin{aligned}\label{E-8-2}
& \Phi(t^*_{(r-1)(N-1)+i}, t^*_{(r-1)(N-1)+i-1}) \geq e^{-\kappa (t_{\ell_{q+1}}-t_{\ell_1})}
\left(I+ \frac{a}{N}\underline{A}_{k_q}\right)\cdots
\left(I+
  \frac{a}{N}
\underline{A}_{k_1}\right)\\
&\hspace{1cm} \geq e^{-\kappa (t^*_{(r-1)(N-1)+i}- t^*_{(r-1)(N-1)+i-1})}  \frac{a}{N}(\underline{A}_{k_q}+\cdots
+\underline{A}_{k_1}).
\end{aligned}
\end{align}
Here, one has

\begin{equation}\label{E-9}
\underline{A}_{k_q}+\cdots+\underline{A}_{k_1}\geq \phi(x^\infty)F_{i}.
\end{equation}
Now, we combine \eqref{E-8-2} with \eqref{E-9} to obtain
\[ \Phi(t^*_{(r-1)(N-1)+i}, t^*_{(r-1)(N-1)+i-1}) \geq e^{-\kappa (t^*_{(r-1)(N-1)+i}- t^*_{(r-1)(N-1)+i-1})} \frac{a}{N}
\phi(x^\infty)F_{i}. \]
This implies
\begin{align}
\begin{aligned}\label{E-10-2}
\Phi(t^*_{r(N-1)}, t^*_{(r-1)(N-1)}) &=\prod_{i=1}^{N-1}\Phi(t^*_{(r-1)(N-1)+i}, t^*_{(r-1)(N-1)+i-1})\\
&\geq e^{-\kappa (t^*_{r(N-1)}- t^*_{(r-1)(N-1)})}
\left(  \frac{a}{N}\right)^{N-1}
\phi(x^\infty)^{N-1}\prod_{i=1}^{N-1}F_i.
%&\geq C_1 \zeta(x^\infty)^{N-1}\prod_{i=1}^{N-1}F_i
\end{aligned}
\end{align}
This verifies the claim \eqref{E-4}. Since the union digraph
\[ \G([t^*_{(r-1)(N-1)+i-1}, t^*_{(r-1)(N-1)+i}))=\G(F_i) \]
has a spanning tree, we apply Proposition \ref{P2.1} to see that $F_1F_2\ldots F_{N-1}$ is scrambling and moreover, \eqref{B-1} yields
 \begin{equation} \label{E-10-3}
 \mu\left(\prod_{i=1}^{N-1}F_i\right)\geq1.
 \end{equation}
Hence, we use \eqref{B-1-1} and \eqref{E-10-3} to get
 \[
 \mu\left(\Phi(t^*_{r(N-1)}, t^*_{(r-1)(N-1)})\right)\geq  e^{-\kappa (t^*_{r(N-1)}- t^*_{(r-1)(N-1)})}
\left(  \frac{a}{N}\right)^{N-1}
\phi(x^\infty)^{N-1}.
 \]
This verifies the relation \eqref{D-6}. 

\vspace{0.4cm}

\noindent For the first assertion,  $\Phi(t^*_{r(N-1)}, t^*_{(r-1)(N-1)})$ is nonnegative by \eqref{E-10-2}. So it remains to show that each of its rows sums to 1. Note that the constant state $\xi(t):=[\xi_1(t),\cdots,\xi_N(t)]^\top\equiv [1,\cdots,1]^\top$ is a solution to $\eqref{D-2}$:
\[
\frac{d}{dt} {\xi}(t) =-\frac{1}{N} L_{\sigma_t}(t) {\xi}(t).
\]
Hence,
\[
[1,\cdots,1]^\top=\Phi(t^*_{r(N-1)}, t^*_{(r-1)(N-1)})[1,\cdots,1]^\top.
\]
This implies that $\Phi(t^*_{r(N-1)}, t^*_{(r-1)(N-1)})$ is stochastic.

\end{proof}

\begin{proposition}\label{P4.1}
\emph{(A priori velocity alignment)} Suppose that the a priori assumptions $({\mathcal P}1) - ({\mathcal P}2)$ hold. Then, for  all $t\in[t^*_{r(N-1)},t^*_{(r+1)(N-1)})$ with  $r \in\bbn$, we have
\begin{align*}
\begin{aligned}
&\D\left(V(t)\right) \leq \D({V}(0))  \exp\left[ -\left(  \frac{a\phi(x^\infty)e^{-\kappa bn}(N-1)^{-\kappa bc}}{N}\right)^{N-1}\frac{(r+1)^{1-\kappa b (N-1) c } -1 }{1-\kappa b (N-1) c  }\right].
\end{aligned}
\end{align*}
\end{proposition}
\begin{proof}
Since $\Phi(t^*_{r(N-1)}, t^*_{(r-1)(N-1)})$ is stochastic (Lemma \ref{L4.1}), we combine Lemma \ref{L2.1} and Lemma \ref{L4.1} to obtain that for $t\in[t^*_{r(N-1)},t^*_{(r+1)(N-1)})$,
\begin{align*}
\begin{aligned}
\D\left(V(t)\right)&\leq\D\left(V(t^*_{r(N-1)})\right)\leq\left[ 1-\mu\Big(\Phi(t^*_{r(N-1)}, t^*_{(r-1)(N-1)})\Big)\right ] \D(V( t^*_{(r-1)(N-1)}))\\
& \leq \left[ 1-e^{-\kappa (t^*_{r(N-1)}- t^*_{(r-1)(N-1)})}
\left(  \frac{a\phi(x^\infty)}{N}\right)^{N-1}
 \right ] \D(V( t^*_{(r-1)(N-1)}))\\
&\leq \exp\left[ -e^{-\kappa (t^*_{r(N-1)}- t^*_{(r-1)(N-1)})}
\left(  \frac{a\phi(x^\infty)}{N}\right)^{N-1}\right ] \D(V( t^*_{(r-1)(N-1)}))\\
&\leq\cdots\leq \exp\left[ -\left(  \frac{a\phi(x^\infty)}{N}\right)^{N-1}\sum_{i=1}^r e^{-\kappa (t^*_{i(N-1)}- t^*_{(i-1)(N-1)})} \right] \D({V}(0))\\
&\leq\exp\left[ -\left(  \frac{a\phi(x^\infty)}{N}\right)^{N-1}\sum_{i=1}^r e^{-\kappa b (a_{i(N-1)}(n,c)- a_{(i-1)(N-1)}(n,c))} \right ] \D({V}(0))\\
&=\exp\left[ -\left(  \frac{a\phi(x^\infty)}{N}\right)^{N-1}\sum_{i=1}^r e^{-\kappa b \left( (N-1)n+\sum_{j=(i-1)(N-1)+1}^{i(N-1)}\lfloor c \log j\rfloor\right)} \right ] \D({V}(0)) \\
&\leq \exp\left[ -\left(  \frac{a\phi(x^\infty)}{N}\right)^{N-1}\sum_{i=1}^r e^{-\kappa b (N-1)\left( n+c \log (i(N-1))\right)} \right ] \D({V}(0))\\
&= \exp\left[ -\left(  \frac{a\phi(x^\infty)e^{-\kappa bn}(N-1)^{-\kappa bc}}{N}\right)^{N-1}\sum_{i=1}^r i^{-\kappa b (N-1) c   } \right ] \D({V}(0))\\
&\leq \exp\left[ -\left(  \frac{a\phi(x^\infty)e^{-\kappa bn}(N-1)^{-\kappa bc}}{N}\right)^{N-1}\int_{1}^{r+1} x^{-\kappa b (N-1) c   } dx\right ] \D({V}(0))\\
&\leq \exp\left[ -\left(  \frac{a\phi(x^\infty)e^{-\kappa bn}(N-1)^{-\kappa bc}}{N}\right)^{N-1}\frac{(r+1)^{1-\kappa b (N-1) c } -1 }{1-\kappa b (N-1) c  }\right ] \D({V}(0)).
\end{aligned}
\end{align*}
\end{proof}

Next, we assert that our a priori condition on the uniform boundedness of distances between CS particles can be obtained from other existing a priori conditions. Before we move on, we present a technical lemma.

\begin{lemma}\label{L4.2}
For any $x>0$ and $\delta>0$, we have the following inequality:
\[
e^{-x}\leq \bigg(\frac{\delta}{e}\bigg)^\delta x^{-\delta}.
\]
\end{lemma}
\begin{proof}
By differentiation, we can check that the function $x\mapsto -x+\delta\log x$ attains its maximal value at $x=\delta$. Hence
\[
-x+\delta\log x\leq -\delta+\delta\log\delta,\quad x>0,~\delta>0.
\]
We take the exponential of both sides to get
\[
e^{-x}x^\delta\leq e^{-\delta}\delta^\delta \quad\Longrightarrow\quad e^{-x}\leq \bigg(\frac{\delta}{e}\bigg)^\delta x^{-\delta}.
\]
\end{proof}
Next, we show that the a priori assumption $({\mathcal P}2)$ on the position diameter can be replaced by the condition on the initial data so that we can establish pathwise flocking estimate under the a priori condition on the network topology. 
\begin{proposition}\label{P4.2}
Suppose that  a priori condition $({\mathcal P}1)$ holds, and  there exist $\delta>0$ and $x^\infty>0$ independent of a sample point such that
\begin{align}\label{C-100}
\begin{aligned}
&\D(X(0)) +\D({V}(0))   b(N-1)(n+c\log ( (N-1)))   \\
&\qquad +\D({V}(0)) b(N-1)\bigg(\frac{\delta}{e }\bigg)^{\delta}\left(  \frac{a\phi(x^\infty)e^{-\kappa bn}(N-1)^{-\kappa bc}}{N}\right)^{-(N-1)\delta}\\
&\qquad\quad\times \sum_{r=1}^\infty\Bigg((n+c\log ((r+1)(N-1)))\left(\frac{(r+1)^{1-\kappa b (N-1) c } -1 }{1-\kappa b (N-1) c  }\right)^{-\delta}\Bigg) < x^{\infty},
\end{aligned}
\end{align}
and let $(X,V)$ be a solution process to \eqref{CS}. Then, a priori condition $({\mathcal P}2)$ holds: for $\omega \in \Omega$, 
\[  \sup_{0 \leq t < \infty} \mathcal D(X(t,\omega))<x^\infty. \]
\end{proposition}
\begin{proof}
We use a contradiction argument for the desired estimate. For this, we define a set ${\mathcal T}$ and its supremum as follows:
\[
{\mathcal T}:=\Big\{T>0:\quad \max_{0 \leq t \leq T} \D(X(t))<x^{\infty} \Big\}, \quad T^*:=\sup {\mathcal T}.
\]
By assumption \eqref{C-100} and the continuity of $\D(X)$, the set ${\mathcal T}$ is nonempty. Now, we claim:
\[ \sup {\mathcal T}=\infty. \]
Suppose not, i.e. $T^*:=\sup {\mathcal T} <\infty$. Then,  we have
\begin{equation} \label{C-101}
\D(X(T^*))=x^{\infty}.
\end{equation}
It follows from Proposition \ref{P4.1} and Lemma \ref{L4.2} that we have
\begin{align*}
x^\infty &= \D(X(T^*))  \leq \D(X(0))+\int_0^{T^*} \D(V(t))dt  \\
& \leq \D(X(0))+\D({V}(0))\sum_{r=0}^\infty\Bigg[ (t^*_{(r+1)(N-1)}-t^*_{r(N-1)})  \\
&\quad\times\exp\left(-\left(  \frac{a\phi(x^\infty)e^{-\kappa bn}(N-1)^{-\kappa bc}}{N}\right)^{N-1}\frac{(r+1)^{1-\kappa b (N-1) c } -1 }{1-\kappa b (N-1) c  }\right)\Bigg] \\
&\leq\D(X(0))+\D({V}(0))\sum_{r=0}^\infty\Bigg[  b(N-1)(n+c\log ((r+1)(N-1)))  \\
&\quad\times\exp\left(-\left(  \frac{a\phi(x^\infty)e^{-\kappa bn}(N-1)^{-\kappa bc}}{N}\right)^{N-1}\frac{(r+1)^{1-\kappa b (N-1) c } -1 }{1-\kappa b (N-1) c  }\right)\Bigg] \\
& \leq\D(X(0)) +\D({V}(0))   b(N-1)(n+c\log ( (N-1)))   \\
&\quad +\D({V}(0))\sum_{r=1}^\infty\Bigg[  b(N-1)(n+c\log ((r+1)(N-1)))\bigg(\frac{\delta}{e }\bigg)^{\delta}\\
&\qquad\times \left(\left(  \frac{a\phi(x^\infty)e^{-\kappa bn}(N-1)^{-\kappa bc}}{N}\right)^{N-1}\frac{(r+1)^{1-\kappa b (N-1) c } -1 }{1-\kappa b (N-1) c  }\right)^{-\delta}\Bigg] \\
& < x^{\infty}.
\end{align*}
This yields a contradiction to \eqref{C-101}. Therefore we have $\sup {\mathcal T}=\infty$.
\end{proof}
As a corollary, we can use Proposition \ref{P4.2} to prove that a priori condition on network structures together with conditions in Theorem \ref{T3.1} implies the uniform boundedness of distances between particles and the velocity relaxation estimates for any initial configuration.
\begin{corollary}\label{C4.1}
Suppose that  a priori condition $({\mathcal P}1)$ holds, and in addition to \eqref{comm}, the communication weight $\phi$ satisfies 
\[ \frac{1}{\phi(|x|)}={\mathcal O}(|x|^{ \varepsilon }) \quad \mbox{as $|x| \to\infty$}, \]
where $\varepsilon$ is a positive constant satisfying the relation $0\leq\varepsilon<\frac{1-\kappa b(N-1)c}{N-1}$. Then the mono-cluster flocking emerges pathwise for any initial configuration: for $\omega \in \Omega$, there exists $x^\infty >0$ such that
\[
\sup_{0\leq t<\infty}\mathcal D(X(t,\omega))\leq x^\infty \quad \mbox{and}\quad \lim_{t\to\infty}\mathcal D(V(t,\omega))=0.
\]
\end{corollary}
\begin{proof} We choose a positive number $\delta>0$ such that 
\begin{equation}\label{C3-1.1}
\frac{1}{1-\kappa b(N-1)c}<\delta<\frac{1}{(N-1)\varepsilon}.
\end{equation}
The left-hand side in \eqref{C3-1.1} implies
\[
\sum_{r=1}^\infty\Bigg[ (n+c\log ((r+1)(N-1)))\left(\frac{(r+1)^{1-\kappa b (N-1) c } -1 }{1-\kappa b (N-1) c  }\right)^{-\delta}\Bigg ] <\infty.
\]
Moreover, the right-hand side in \eqref{C3-1.1} implies 
\[ \phi(|x|)^{-(N-1)\delta}=O(|x|^{ (N-1)\delta \varepsilon }) \quad \mbox{as $|x| \to\infty$}. \]
Hence, one has
\[
\lim_{|x| \to\infty}\frac{\phi(|x|)^{-(N-1)\delta}}{|x|}=0.
\]
This implies the existence of $x^\infty$ satisfying \eqref{C-100} for $\delta$ chosen in \eqref{C3-1.1}. Hence the condition \eqref{C-100} is satisfied, and the results follow from Proposition \ref{P4.1} and Proposition \ref{P4.2}.
\end{proof}

\subsection{Emergence of stochastic flocking} \label{sec:4.3} In the previous section, we verified the emergence of pathwise flocking under the a priori assumption on the network structure $({\mathcal P}1)$. In the sequel, we will show that the a priori assumption  $({\mathcal P}1)$ can be guaranteed with probability one.  \newline

Next step is to prove that {\it a priori} assumption $({\mathcal P}1)$ on network structure can be satisfied for most of $\omega\in\Omega$, once we determine appropriate values for $n$ and $c$.
\begin{proposition}\label{P4.3}
Let $(X,V)$ be a solution process to \eqref{CS}, and let $n\in\bbn$ and $c>0$ be such that
\[ \sum_{k=1}^{N_G} (1-p_k)^{n}\leq\frac{1}{2} \quad\mbox{and}\quad c>   \frac{1}{\min\limits_{1 \le k \le N_G}\log \frac{1}{(1-p_k)}}. \]
Then, the following assertions hold.
\begin{enumerate}
\item
The subsequence $\{t_\ell^*\}_{\ell\in\bbn}\subset\{t_\ell\}_{\ell\in\bbn}$, defined by $t_\ell^*:= t_{a_\ell(n,c)}$ in \eqref{D-3-1}, satisfies
\begin{align*}
\begin{aligned}
& \mathbb{P} \Big( \omega : \mathcal{G}( [t_\ell^*, t_{\ell+1}^*))(\omega) \mbox{ has a spanning tree for any  $\ell \ge 0$} \Big ) \\
& \hspace{2cm}  \geq \exp\left( -(2\log 2) \sum_{k=1}^{N_G} (1-p_k)^n\sum_{\ell=0}^\infty  (1-p_k)^{\lfloor c \log (\ell+1)\rfloor}\right).
\end{aligned}
\end{align*}
\item
The serie $\sum\limits_{\ell=0}^\infty  (1-p_k)^{\lfloor c\log (\ell+1)\rfloor}$ converges for all $k=1,\cdots, N_G$. 
\end{enumerate}
\end{proposition}
\begin{proof}
\noindent (i)~For any $q$, $r\in\bbn$, we have the following estimate:
\begin{align*}
\begin{aligned}
&\mathbb{P}(\omega : \mathcal{G}([t_q, t_{q+r}))(\omega) \mbox{ does not have a spanning tree})\\
& \hspace{0.5cm} \le \mathbb{P}(\omega :\exists 1\leq k\leq N_G~~\mbox{ such that } \sigma_{t_{q+i}}(\omega)\neq k \mbox{ for }\forall~ 0\le i\le r-1)\\
& \hspace{0.5cm} \le \sum_{k=1}^{N_G} \mathbb{P}(\omega : \sigma_{t_{q+i}}(\omega)\neq k \mbox{ for }\forall~ 0\le i\le r-1) = \sum_{k=1}^{N_G} (1-p_k)^r,
\end{aligned}
\end{align*}
where the last inequality follows from the independence of $\{t_{\ell+1}-t_\ell\}_{\ell\in\bbn}$. \newline

This implies
\[
\mathbb{P}(\omega : \mathcal{G}([t_q, t_{q+r}))(\omega) \mbox{ has a spanning tree})\ge 1-\sum_{k=1}^{N_G} (1-p_k)^r.
\]
Here, we substitute $a_\ell(n,c)$ and $a_{\ell+1}(n,c)$ for $q$ and $q+r$, respectively, and take the product over $\ell\in\bbn$ to see the following relations:
\begin{align*}
\begin{aligned}
&\mathbb{P}(\omega : \mathcal{G}([t_\ell^*, t_{\ell+1}^*))(\omega) \mbox{ has a spanning tree for any } \ell \in\bbn)\\
& \hspace{0.2cm} = \prod_{\ell=0}^\infty \mathbb{P}(\omega : \mathcal{G}([t_\ell^*, t_{\ell+1}^*))(\omega) \mbox{ has a spanning tree }) \ge \prod_{\ell=0}^\infty \left(1-\sum_{k=1}^{N_G} (1-p_k)^{a_{\ell+1}-a_\ell}\right)\\
& \hspace{0.2cm}  = \exp\left(\sum_{\ell=0}^\infty \log\left(1-\sum_{k=1}^{N_G} (1-p_k)^{n+\lfloor c \log (\ell+1)\rfloor}\right) \right) \ge \exp\left( -(2\log 2) \sum_{\ell=0}^\infty \sum_{k=1}^{N_G} (1-p_k)^{n+\lfloor c \log (\ell+1)\rfloor}\right)\\
& \hspace{0.2cm} = \exp\left( -(2\log 2) \sum_{k=1}^{N_G} (1-p_k)^n\sum_{\ell=0}^\infty  (1-p_k)^{\lfloor c\log (\ell+1)\rfloor}\right),
\end{aligned}
\end{align*}
where we used the following inequality:
\[\log(1-x) \ge -(2\log 2)x , \quad 0 \le x \le \frac{1}{2}. \]

\vspace{0.2cm}

\noindent (ii)~The convergence of the series $\sum\limits_{\ell=0}^\infty  (1-p_k)^{\lfloor c \log (\ell+1)\rfloor}$ can be shown as follows: by comparison test, it suffices to show that
\[
\sum_{\ell=0}^\infty  (1-p_k)^{c\big(\log (\ell+1)-1\big)}<\infty\quad\Longleftrightarrow\quad\sum_{\ell=1}^\infty  (1-p_k)^{c\log \ell}<\infty.
\]
By Cauchy's condensation test, the right-hand side of the above is equivalent to
\[
\sum_{\ell=1}^\infty  2^\ell(1-p_k)^{c\log (2^\ell)}=\sum_{\ell=1}^\infty  \big(2(1-p_k)^{c\log 2}\big)^\ell<\infty.
\]
The condition $c>\frac{1}{\log \frac{1}{(1-p_k)}}$ is equivalent to $0<2(1-p_k)^{c\log2}<1$. Thus, we have the desired result.
\end{proof}

\vspace{0.2cm}

\noindent {\bf The proof of Theorem \ref{T3.1}}: We choose $\varepsilon$ to satisfy
\[ \displaystyle 0 \le \e < \frac{1}{N-1} - \frac{\kappa b}{\min\limits_{1\le k \le N_G} \log\frac{1}{1-p_k}}, \quad \mbox{or equivalently} \quad \frac{1}{\min\limits_{1\le k \le N_G} \log\frac{1}{1-p_k}} <  \frac{1-\e (N-1)}{\kappa b (N-1)}, \]
and we set
\[
c:= \frac{1}{2} \left( \frac{1}{\min\limits_{1\le k \le N_G} \log\frac{1}{1-p_k}} + \frac{1-\e (N-1)}{\kappa b (N-1)} \right).
\]
Then, it is easy to see that the constant $c$ defined above satisfies

\[
c > \frac{1}{\min\limits_{1\le k \le N_G} \log\frac{1}{1-p_k}}, \qquad \kappa b (N-1)c <1, \qquad \mbox{and} \qquad  0 \le \e < \frac{1-\kappa b(N-1)c}{N-1}.
\]

\noindent Now, we choose any $n \in \bbn$ such that
\[ \sum_{k=1}^{N_G} (1-p_k)^{n}\leq\frac{1}{2},  \]
and we define $p(n)$ as

\[
p(n) := \exp\left[  -(2\log 2) \sum_{k=1}^{N_G} (1-p_k)^n\sum_{\ell=0}^\infty  (1-p_k)^{\lfloor c \log (\ell+1)\rfloor} \right ].
\]
With this choice of $n$ and $c$, Proposition \ref{P4.3} implies
\[\mathbb{P}\Big \{ \omega : \mathcal{G}([t_\ell^*, t_{\ell+1}^*))(\omega) \mbox{ has a spanning tree for any } \ell \ge 0\Big \} \ge p(n), \quad t_\ell^* := t_{a_l(n,c)}. \]
Hence, it follows from Corollary \ref{C4.1} that
\[\mathbb{P}\Big\{ \omega : \exists \ x^\infty>0 \mbox{ s.t } \sup_{0\le t < \infty }\D(X(t,\omega)) \le x^\infty, \mbox{ and } \lim_{t \to \infty} \D(V(t,\omega)) =0 \Big \} \ge p(n). \]
Since $n$ can be arbitrarily large and $p(n) \to 1$ as $n \to \infty$, our desired result follows.

%By Corollary \ref{C3.1}, it is sufficient to show the following:
%\[
%\mathbb P\left(\exists 0<c<\frac{1-(N-1)\varepsilon}{\kappa b(N-1)},~\exists n\in\bbn ~\mbox{s.t.}~\mathcal G([t_{a_\ell(n,c)},t_{a_{\ell+1}(n,c)}))~\mbox{has a spanning tree for}~ \forall \ell\geq0\right)=1.
%\]
%Note that by Proposition \ref{P3.3} we obtain
%\begin{align*}
%\begin{aligned}
%&\mathbb P\left(\exists 0<c<\frac{1-(N-1)\varepsilon}{\kappa b(N-1)},~\exists n\in\bbn ~\mbox{s.t.}~\mathcal G([t_{a_\ell(n,c)},t_{a_{\ell+1}(n,c)}))~\mbox{has a spanning tree for}~ \forall \ell\geq0\right)\\
%&\geq \sup_{\frac{\kappa b(N-1)}{\min\limits_{1\leq k\leq \gamma}\log\frac{1}{1-p_k}}<c<\frac{1-(N-1)\varepsilon}{\kappa b(N-1)},n\in\bbn}\mathbb P\left(\mathcal G([t_{a_\ell(n,c)},t_{a_{\ell+1}(n,c)}))~\mbox{has a spanning tree for}~\forall \ell\geq0\right)\\
%&\geq \sup_{\frac{\kappa b(N-1)}{\min\limits_{1\leq k\leq \gamma}\log\frac{1}{1-p_k}}<c<\frac{1-(N-1)\varepsilon}{\kappa b(N-1)},n\in\bbn} \exp\left( -(2\log 2) \sum_{k=1}^\gamma(1-p_k)^n\sum_{\ell=0}^\infty  (1-p_k)^{c\lfloor \log (\ell+1)\rfloor}\right)\\
%&\geq \sup_{\frac{\kappa b(N-1)}{\min\limits_{1\leq k\leq \gamma}\log\frac{1}{1-p_k}}<c<\frac{1-(N-1)\varepsilon}{\kappa b(N-1)}} \exp\left( 0\right)=1.
%\end{aligned}
%\end{align*}
%This finishes the proof.

\vspace{0.4cm}

\section{Conclusion}\label{sec:5}
In this paper, we have introduced the Cucker-Smale model with  randomly switching topologies for flocking phenomenon, and provided a sufficient framework leading to the stochastic flocking in terms of system parameters and communication weight function. For the stochastic flocking modeling, we employed two random components for the switching times and selection of network topology at switching instant. Our flocking analysis took two procedures: First, we derived flocking estimates along the sample path in a priori setting on the network topologies and position diameter. Second, we replaced a priori assumption on the position diameter by suitable assumptions on the system parameters and communication weight, and moreover, we also showed that the a priori assumption on the network topology can be attained by imposing some condition on the network selection probability. There are still many questions to be investigated for the proposed model. For example, what if the support of the probability density function $f$ for the sequence $\{t_{\ell+1} - t_\ell\}_{\ell \geq 0}$ is not compactly supported, say $(0,\infty)$? Clearly, our analysis employed in the proof of main result breaks down for unbounded support case. However, it seems that our methodology and framework is quite general so that it can be applied to other CS type flocking and Kuramoto type synchronization models. These issues will be addressed in our future works.

\end{document}